\documentclass[runningheads,twoside]{llncs}

\usepackage[ansinew]{inputenc}
\usepackage[T1]{fontenc}

\usepackage{graphicx}
\usepackage{amssymb}
\usepackage{textcomp}
\usepackage{amscd}
\usepackage{amsmath}
\usepackage[usenames]{color}
\definecolor{sonnengelb}{rgb}{0.9,0.5,0} 
\definecolor{brown}{rgb}{0.5,1,0}
\usepackage{algorithm,algorithmic}

\newtheorem{Problem}{Problem}[section]

\begin{document}
\mainmatter              
\title{A Note on the Characterization of Digraph Sequences \thanks{This work
    was supported by the DFG Focus Program Algoritheorem Engineering,
    grant MU 1482/4-2.}}
\titlerunning{Digraph Sequences}
\author{Annabell Berger}
\institute{Department of Computer Science\\
  Martin-Luther-Universit\"at Halle-Wittenberg\\
\email{berger@informatik.uni-halle.de}}
\authorrunning{A. Berger}
\maketitle

\begin{abstract}
We consider the following fundamental realization problem of directed graphs.
Given a sequence $S:={a_1 \choose b_1},\dots,{a_n \choose b_n}$ with $a_i,b_i\in \mathbb{Z}_0^+.$ Does there exist a digraph (no loops and no parallel arcs are allowed)
$G=(V,A)$ with a labeled vertex set $V:=\{v_1,\dots,v_n\}$ such that for all $v_i \in V$ indegree and outdegree of $v_i$ match exactly the given numbers $a_i$ and $b_i$,
respectively? There exist two known approaches solving this problem in polynomial running time. One first approach of Kleitman and Wang (1973) uses recursive algorithms to construct digraph realizations \cite{KleitWang:73}. The second one draws back into the Fifties and Sixties of the last century and gives a complete characterization of digraph sequences (Gale 1957, Fulkerson 1960, Ryser 1957, Chen 1966). That is, one has only to validate a certain number of inequalities. Chen bounded this number by $n$. His characterization demands the property that $S$ has to be in lexicographical order. We show that this condition is stronger than necessary. We provide a new characterization which is formally analogous to the classical one by Erd{\H o}s and Gallai (1960) for graphs. Hence, we can give several, different sets of $n$ inequalities. We think that this stronger result can be very important with respect to structural insights about the sets of digraph sequences, for example in the context of threshold sequences. Furthermore, the number of inequalities can be restricted to all $k \in \{1,\dots,n-1\}$ with $a_{k+1}>a_{k}$ and to $k=n.$ An analogous result for graphs was given by Tripathi and Vijay \cite{TripathiVijay03}. We prove this property also for the case of digraphs (no parallel arcs) with at most one loop per vertex.
\end{abstract}

\section{Characterization of Digraph Sequences}

\begin{Problem}[digraph realization problem]\label{problem:digraph-realization}
Given is a finite sequence $S:={a_1 \choose b_1},\dots,{a_n \choose b_n}$
with $a_i,b_i\in \mathbb{Z}_0^+.$ Does there exist a digraph (without parallel arcs and loops)
$G=(V,A)$ with a labeled vertex set $V:=\{v_1,\dots,v_n\}$ such that
we have indegree $d^-_{G}(v_i)=a_i$ and outdegree $d^+_{G}(v_i)=b_i$ for all $v_i\in V$?
\end{Problem}

If the answer is ``yes'', we call sequence $S$ \emph{digraph sequence} and  digraph $G$ a \emph{digraph realization}. We exclude tuples ${0 \choose 0}$ in $S.$ Furthermore, we will tacitly assume that $\sum_{i=1}^{n}a_i=\sum_{i=1}^{n}b_i$, as this is obviously a necessary condition for any realization to exist, since the number of ingoing arcs must equal the number of outgoing arcs. We say that sequence $S$ is in \emph{decreasing lexicographical order} if we have ${a_i \choose b_i}\geq_{lex}{a_{i+1} \choose b_{i+1}}$ for all $i \in \{1,\dots,n-1\}$, i.e.\ $(a_i>a_{i+1}) \lor (a_i=a_{i+1} \land b_i \geq b_{i+1}).$ There exist two known approaches solving this problem in polynomial running time. One first approach of Kleitman and Wang uses recursive algoritheorems to construct digraph realizations \cite{KleitWang:73}. The second one gives a complete characterization of digraph sequences. Hence, it is possible to check a polynomial number of inequalities (in the size of the number of tuples in a sequence) which leads to the correct decision of the realizability of a sequence. Note that the digraph realization problem allows opposite directed arcs (i.e., $(x,y),(y,x) \in A$). As far as we know the analogous problem for oriented digraphs is open. However, these views come from another analogous problem --- the \emph{graph realization problem} which asks whether a given \emph{undirected sequence} $S:=(d_1),\dots,(d_n)$ with $d_i\in \mathbb{Z}_0^+$ possesses a realization as a graph. The characterization approach was found by Erd{\H o}s and Gallai \cite{ErGa:60} and a realization algorithm was introduced by Havel \cite{Havel:56} and Hakimi \cite{Hakimi:62}. Four authors, namely David Gale \cite{Ga:57}, Herbert J. Ryser \cite{Ry:57}, Delbert Ray Fulkerson \cite{Fulk:60} and Wai-Kai Chen \cite{Chen:66} gave sufficient and necessary conditions which completely characterize digraph sequences. Actually, none of the mentioned authors has found the following theorem in its general form. Gale and Ryser dealt with digraph sequences where at most one loop per vertex but bo parallel arc is allowed. We call their problem \emph{digraph realization problem with at most one loop per vertex}. Note, that their version of the realization problem also characterizes vertex degree sequences of bipartite graphs. The reason is that the adjacency matrix of a digraph with at most one loop per vertex corresponds to the adajcency matrix of a bipartite graph. The row sets and column sets can be considered as the two independent vertex sets. However, their insights are also fundamental for characterizing digraph sequences without loops. Fulkerson gave a general result with exponentially many inequalities and Chen formulated and proved the final details. Therefore, we cite the following theorem as the result of all four authors.

\begin{theorem}[Characterization of digraph sequences \cite{Fulk:60,Ry:57,Chen:66,Ga:57}] \label{Theorem: Charakterisierung von Digraph-Sequenzen}
Let $S:={a_1 \choose b_1},\dots,{a_n \choose b_n}$ be a sequence in decreasing lexicographical order. $S$ is a digraph sequence if and only if the following conditions are fulfilled for each $k \in \{1,\dots,n\}$:
\begin{gather}\sum_{i=1}^{k}\min{(b_i,k-1)}+\sum_{i=k+1}^{n}\min{(b_i,k)}\geq \sum_{i=1}^{k}a_i.\tag{*}\end{gather}
\end{theorem}

The necessity is easy to see. Consider an arbitrary digraph $G$ with labeled vertex set $\{v_1,\dots,v_n\}.$ We count all ingoing arcs of the first $k$ vertices. This value corresponds to $\sum_{i=1}^{k}a_i.$ Either these ingoing arcs are (a) outgoing arcs between the first $k$ arcs or (b) directed arcs from the remaining vertices $v_{k+1},\dots,v_n$ to the first $k$ arcs. The number of arcs in case (a) is at most $\sum_{i=1}^{k}\min(b_i,k-1),$ because the number of arcs between one vertex to all other vertices in this subset is at most $k-1,$ unless the outdegree is smaller than $k-1.$ In this case the number of arcs between one vertex to all other vertices is at most its outdegree. Using a similar argument the number of arcs in case (b) is at most $\sum_{i=k+1}^{n}\min(b_i,k.).$ This leads to the conditions in (*). Note, that it is not necessary to sort the vertices in decreasing lexicographical order to get necessity of the theorem. Hence, the necessity is fulfilled for each ordering of a sequence $S.$ Chen proved the sufficiency of this theorem for a decreasing lexicographical sorted sequence. One could ask if it is also fulfilled for each sorting of a sequence. Unfortunately, this is not true. Consider the following counter example. Sequence $S:={3 \choose 3},{1 \choose 2},{3 \choose 3},{3 \choose 2}$ is not a digraph sequence, which can easy be seen sorting it in a decreasing lexicographical order and validating the conditions in Theorem \ref{Theorem: Charakterisierung von Digraph-Sequenzen}. Assume it is sufficient to use the given ordering of $S.$ Then all conditions of the above theorem are fulfilled. In contrast to the property that necessity is fulfilled for each sorting of a sequence, the set of sortings for suffiency is smaller. It turns out that sufficiency is fulfilled for all sequences with decreasingly sorted components $a_i.$ Furthermore, the number of inequalities can be restricted to all $k \in \{1,\dots,n-1\}$ with $a_{k+1}>a_{k}$ and to $k=n.$ An analogous result for graphs was given by Tripathi and Vijay \cite{TripathiVijay03}. Their proof was simplified by Dahl and Flatberg \cite{DahlFlatberg05} using a simple geometric argument. Note, that semi-regular sequences for digraphs $G=(V,A)$ (one component is regular) are only needed to be checked in (*) for $k=n.$ Since this inequality is always fulfilled for any semi-regular sequence with $b_i \leq n-1,$  such sequences are always realizable, when they fulfill these obvious constraints. In the version of Ryser and Gale (directed graphs with at most one loop per vertex, bipartite graphs), the left-hand-side of (*) is the sum $\min(b_i,k)$ over all vertices. For each sorting of a sequence with decreasingly sorted $a_i $ these sums  are constant for fixed $k.$  Hence, the reformulated result as given in this paper holds trivially for the case of digraph realizations with loops. To the best of our knowledge, there does not exist a result restricting the number of inequalities to all $k$ with $a_{k+1}<a_{k}$ analogous to the insight of Tripathi and Vijay. Since, it is very simple to verify this property we add its proof to fill this gap in Theorem \ref{Theorem:strictMajorization}. In contrast to this version, the left-hand-side of (*) in Theorem \ref{Theorem: Charakterisierung von Digraph-Sequenzen} can be different for several sortings of sequences with decreasingly sorted $a_i$ and fixed $k.$ In this case, Chen's idea says that it is sufficient to order the $b_i$ from the biggest to smallest, because the left sum is then minimum. We relax the conditions of this theorem in the sense showing that sequence $S$ has only to be sorted in decreasing order with respect to its first components $a_i.$ We think that this observation could be very important in the context of \emph{threshold sequences} (see the book by Mahadev and Peled\cite{MaPe:95}). A threshold sequence can be defined as the unique sequence fulfilling all inequalities in Theorem \ref{Theorem: Charakterisierung von Digraph-Sequenzen} with equality. Hence, our result leads to several and not only one threshold sequence for a given sequence $S$ (dependent on the number of decreasing permutations of $S$). The corresponding adjacency matrix of the digraph realization of a threshold sequence is also known as \emph{Ferrers matrix}. This classical relationship has been pointed out in \cite{Be:11}. Ferrers `diagrams' were first given by Sylvester \cite{Syl:1882} in 1882.  We need a simple and well-known combinatorial insight using the principle of double counting in a Ferrers matrix. Counting the entries `one' of the first $k$ columns there is an equivalent formulation of this sum with respect to the entries in their rows. We get for an arbitrary, finite set $M \subset \mathbb{N}$, $k \in \mathbb{N}:$  
\begin{gather}\sum_{i \in M} \min{(i,k)}=\sum_{j=1}^{k}|\{i \in M|~i\geq j\}|. \end{gather}

We prove a slightly stronger version of the theorem of Ryser and Gale \cite{Ry:57,Ga:57}.

\begin{theorem}\label{Theorem:strictMajorization}
Sequence $S:={a_1 \choose b_1},\dots,{a_n \choose b_n}$ with $a_1 \geq \dots \geq a_n$ is a digraph sequence with at most one loop per vertex if and only if we find for all $k \in \{1,\dots,n-1\}$ with $a_k >a_{k+1}$ and for $k=n$ that
\begin{gather}\sum_{i=1}^{n}\min{(b_i,k)}\geq \sum_{i=1}^{k}a_i.\tag{**}\end{gather}
\end{theorem}

\begin{proof}
The theorem was proven for all $k \in \{1,\dots,n\}$ by Ryser and Gale \cite{Ry:57,Ga:57}. We define for each sequence $S$ the left-hand-side sum by
$X(k):=\sum_{i=1}^{n}\min{(b_i,k)}$ for $k\in \{1,\dots,n-1\}.$ Furthermore, we set $A_k:=\sum_{i=1}^{k}a_i$ and define $X(0):=0.$
Using (1) it is easy to see that  $X(k+1)-X(k)=|\{b_i|~b_i\geq k+1\}|\leq |\{b_i|~b_i\geq k\}|=X(k)-X(k-1)$ for all $k \in \{0,\dots,n\}.$ 
We assume for all indices $i$ with $a_{i+1}>a_{i}$ that $X(i)\geq A_i.$ Let $k$ be the first index such that $a_k=a_{k+1}$ and $X(k)<A_k.$ Let $k' >k$ the first index with $a_{k'}<a_k.$ If there does not exist such a $k',$ we set $k':=n.$ With our assumption we have $X(k')\geq A_{k'}.$ Since $X(k-1) \geq A_{k-1}$, we get $X(k)-X(k-1)<a_k$ if $k>1.$ For $k=1$ we have $X(1)-X(0)<a_k.$ It follows for each $i$ with $k<i<k',$ 
$$X(i)-X(i-1)\leq X(k)-X(k-1)<a_k.$$
Furthermore, we have 
\begin{eqnarray*}
X(k')-X(k)&=&(X(k')-X(k'-1))+(X(k'-1)-X(k'-2))+\dots\\
& & + (X(k+2)-X(k+1))+(X(k+1)-X(k))\\
&<&(k'-k)a_k.
\end{eqnarray*}
We get $X(k')=X(k')-X(k)+X(k)<(k'-k)a_k+A_k+a_{k'}=A_{k'}.$\hfill$\Box$
\end{proof}

To prove our main result for digraph sequences, we need some further insights. We denote an ordering of a given sequence $S$ by $S_{\sigma}:={a_{\sigma(1)} \choose b_{\sigma(1)}}\dots,{a_{\sigma(n)} \choose b_{\sigma (n)}},$ where $\sigma: \mathbb{N}_n \mapsto \mathbb{N}_n$ is a permutation. For a simpler notion, we denote the left-hand-side of (*) in Theorem \ref{Theorem: Charakterisierung von Digraph-Sequenzen} for an arbitrary ordering $S_{\sigma}$ of sequence $S$ by $$X_{\sigma}(k):=\sum_{i=1}^{k}\min{(b_{\sigma (i)},k-1)}+\sum_{i=k+1}^{n}\min{(b_{\sigma (i)},k)}$$ for $k \in \{1,\dots,n\}.$ Furthermore, we define $X_{\sigma}(0):=0$ for all permutations $\sigma.$ 

\begin{proposition}\label{ColumnSumInsightVorbereitung}
For each ordering $S_{\sigma}$ of $S$ and $k \in \{0,1,\dots,n-1\}$ we have
$$X_{\sigma}(k+1)-X_{\sigma}(k)=|\{b_{\sigma (i)}|~i\leq k, b_{\sigma (i)} \geq k\}|+|\{b_{\sigma (i)}|~i\geq k+2,b_{\sigma (i)} \geq k+1\}|.$$
\end{proposition}

\begin{proof}
For $k \geq 1$ we get
\small
\begin{eqnarray*}
X_{\sigma}(k+1)-X_{\sigma}(k)&=&\sum_{i=1}^{k+1}\min{(b_{\sigma (i)},k)}+\sum_{i=k+2}^{n}\min{(b_{\sigma (i)},k+1)}\\
&& -\left(\sum_{i=1}^{k}\min{(b_{\sigma (i)},k-1)}+\sum_{i=k+1}^{n}\min{(b_{\sigma (i)},k)}\right)\\
&=&\sum_{i=1}^{k}\min{(b_{\sigma (i)},k)}+\min{(b_{\sigma(k+1)},k)}+\sum_{i=k+2}^{n}\min{(b_{\sigma (i)},k+1)}\\
& & - \left(\sum_{i=1}^{k}\min{(b_{\sigma (i)},k-1)}+\min{(b_{\sigma (k+1)},k)}+\sum_{i=k+2}^{n}\min{(b_{\sigma (i)},k)}\right)\\
&=&\sum_{i=1}^{k}\min{(b_{\sigma (i)},k)}-\sum_{i=1}^{k}\min{(b_{\sigma (i)},k-1)}\\
& &+\sum_{i=k+2}^{n}\min{(b_{\sigma (i)},k+1)}-\sum_{i=k+2}^{n}\min{(b_{\sigma (i)},k)}\\
&\overset{(1)}{=}&|\{b_{\sigma (i)}|~i\leq k, b_{\sigma (i)} \geq k\}|
+|\{b_{\sigma (i)}|~i \geq k+2,b_{\sigma (i)} \geq k+1\}|.\\
\end{eqnarray*}
\normalsize
For $k=0,$ we have 
$X_{\sigma}(1)-X_{\sigma}(0)=\sum_{i=2}^{n}\min{(b_{\sigma (i)},1)}\overset{(1)}{=}|\{b_{\sigma (i)}|~i \geq 2,b_{\sigma (i)} \geq 1\}|.$\hfill$\Box$
\end{proof}

\begin{proposition}\label{ColumnSumInsight}
For each ordering $S_{\sigma}$ of sequence $S$ with $k',k \in \{1,\dots, n\}$ and $k'>k$  we have 
$$X_{\sigma}(k')-X_{\sigma}(k'-1)\leq X_{\sigma}(k)-X_{\sigma}(k-1)+1.$$
Moreover, if we find $X_{\sigma}(k)-X_{\sigma}(k-1)= X_{\sigma}(k-1)-X_{\sigma}(k-2)+1$ for one $k \geq 2 ,$ then we get for all $k'>k$
$$X_{\sigma}(k')-X_{\sigma}(k'-1)\leq X_{\sigma}(k)-X_{\sigma}(k-1).$$
\end{proposition}

\begin{proof}
\begin{eqnarray*}
X_{\sigma}(k')-X_{\sigma}(k'-1)&\overset{\textnormal{Prop. \ref{ColumnSumInsightVorbereitung}}}{=}&|\{b_{\sigma (i)}|~i\leq k'-1, b_{\sigma (i)} \geq k'-1\}|+|\{b_{\sigma (i)}|~i \geq k'+1,b_{\sigma (i)} \geq k'\}|\\
&= & |\{b_{\sigma (i)}|~i \leq k-1, b_{\sigma (i)} \geq k'-1\}|+|\{b_{\sigma (k)}|~b_k\geq k'-1\}|\\
& &+|\{b_{\sigma (i)}|~k+1\leq i \leq k'-1,b_{\sigma (i)} \geq k'-1\}|+
|\{b_{\sigma (i)}|~i \geq k'+1,b_{\sigma (i)} \geq k'\}|\\
& \leq & |\{b_{\sigma (i)}|~i \leq k-1, b_{\sigma (i)} \geq k-1\}|+|\{b_{\sigma (k)}|~b_k\geq k'-1\}|\\
&& +|\{b_{\sigma (i)}|~i \geq k+1,b_{\sigma (i)} \geq k\}|-|\{b_{\sigma (k')}|~b_{\sigma (k')}\geq k\}|\\
&\overset{\textnormal{Prop. \ref{ColumnSumInsightVorbereitung}}}{\leq}& X_{\sigma}(k)-X_{\sigma}(k-1)+|\{b_{\sigma (k)}|~b_k\geq k'-1\}|-|\{b_{\sigma (k')}|~b_{\sigma (k')}\geq k\}|\\
&\leq & X_{\sigma}(k)-X_{\sigma}(k-1)+1
\end{eqnarray*}
Assume $X_{\sigma}(k)-X_{\sigma}(k-1)= X_{\sigma}(k-1)-X_{\sigma}(k-2)+1,$ for one $k \geq 2.$ Clearly, all inequalities in the previous chain of inequalities must be equalities. Hence, we get in line $4$, $b_{\sigma(k)} < k-1$ and $b_{\sigma(k-1)}\geq k-1.$  For each $k'>k$ it follows that $|\{b_{\sigma (k)}|~b_{\sigma(k)}\geq k'-1\}|=0.$\hfill$\Box$ 
\end{proof}

\noindent We prove the main result of this note.

\begin{theorem}[Digraph Characterization]
Sequence $S:={a_1 \choose b_1},\dots,{a_n \choose b_n}$ with \\$a_1 \geq \dots \geq a_n$ is a digraph sequence if and only if we find for all $k \in \{1,\dots,n-1\}$ with $a_k >a_{k+1}$ and for $k=n$ that
\begin{gather}\sum_{i=1}^{k}\min{(b_i,k-1)}+\sum_{i=k+1}^{n}\min{(b_i,k)}\geq \sum_{i=1}^{k}a_i.\tag{**}\end{gather}
\end{theorem}

\begin{proof}
$\Rightarrow:$ Consider the remarks after Theorem \ref{Theorem: Charakterisierung von Digraph-Sequenzen}.\\
$\Leftarrow:$ We have to show that conditions (**) for sequence $S$ lead to conditions (*) in Theorem~\ref{Theorem: Charakterisierung von Digraph-Sequenzen} for a lexicographical sorting of $S.$ In this case, sequence $S$ is a digraph sequence.\\ 
Let us start with two permutations of $S$, namely $S_{\sigma}$ and $S_{\tau}$ with $a_{\sigma(1)}\geq \dots \geq a_{\sigma(n)}$ and $a_{\tau(1)}\geq \dots \geq a_{\tau(n)}.$ 
These two permutations shall only differ in two adjacent positions $\mu$ and $\mu+1$, $\mu \in \{1,\dots,n-1\}$. Sequence $S_{\sigma}$ shall be one step closer to a lexicographical order than $S_{\tau}$. Since we have $a_{\sigma(\mu)}=a_{\tau(\mu)}=a_{\sigma(\mu+1)}=a_{\tau(\mu+1)},$ we get $b_{\sigma(\mu)}>b_{\tau(\mu)}.$  
Define $A_k:=\sum_{i=1}^{k}a_{\sigma(i)}=\sum_{i=1}^{k}a_{\tau(i)}.$ We show for each $k$:
\begin{gather} X_{\tau}(k)\geq A_k \Rightarrow X_{\sigma}(k) \geq A_k.\tag{+}\end{gather}
For $k\neq \mu$ we have $X_{\tau}(k)=X_{\sigma}(k)\geq A_k.$ For  $k=\mu$ Proposition \ref{ColumnSumInsight} gives:
\begin{eqnarray*}
X_{\sigma}(\mu)&\overset{\textnormal{Prop. \ref{ColumnSumInsight}}}{\geq} &\frac{1}{2}\left(X_{\sigma}(\mu + 1)+X_{\sigma}(\mu -1)-1\right)\\
&=&\frac{1}{2}\left(X_{\tau}(\mu + 1)+X_{\tau}(\mu -1)-1\right)\\
&\geq & \frac{1}{2}\left(A_{\mu}+a_{\tau(\mu +1)}+A_{\mu}-a_{\tau (\mu)}-1\right)\\
&=& A_{\mu}-\frac{1}{2}.\\
\end{eqnarray*}
As $X_{\sigma}(\mu)$ and $A_{\mu}$ are integers, this implies $X_{\sigma}(\mu)\geq A_{\mu}.$  Let us now consider a lexicographical sorted permutation $S_{lex}$ of $S$ and a permutation $S_{\tau}$ of $S$ with decreasing sorted components $a_{\tau(i)}.$ Furthermore, we assume for all $k \in \{1,\dots,n\}$ that we have $X_{\tau}(k)\geq A_k.$ Clearly, it is possible to construct a sequence $S_{\tau}:=S_{1},\dots,S_{\nu}:=S_{lex}$ of permutations $S_{j}$ with $a_{j(1)}\geq \dots \geq a_{j(n)}$ such that two adjacent sortings $S_{j}$ and $S_{j+1}$ do only differ in two adjacent positions $j(\mu) <j(\mu+1).$ Then we can conclude for each $k \in \{1,\dots,n\}$ step by step with our insights above and starting with $X_{1}(k)\geq A_k$ that $X_{2}(k)\geq A_k,\dots, X_j(k)\geq A_k,\dots,X_{lex}(k)\geq A_k.$
In a second step, we show that it is sufficient to consider indices $k$  with $a_{k+1}>a_{k}$ in (**). We assume for all indices $k$ with $a_{k+1}>a_{k}$ that $X_{\sigma}(k)\geq A_k.$ Assume $k_0$ is the first index such that $a_{k_{0}}=a_{k_{0}+1}$ and $X_{\sigma}(k_{0})<A_{k_{0}}.$ Let $k' >k_0$ the first index with $a_{k'}<a_{k_{0}}.$ If there does not exist such a $k'$ we set $k':=n.$ With our assumption we have $X_{\sigma}(k')\geq A_{k'}.$ Since, $X_{\sigma}(k_0-1) \geq A_{k_{0}-1}$ for $k_0>0,$ we get $X_{\sigma}(k_0)-X_{\sigma}(k_0-1)<a_{k_{0}}.$ For $k=1$ we have $X(1)-X(0)<a_k.$ With Proposition \ref{ColumnSumInsight} it follows for each $i$ with $k<i<k'$ 
$$X_{\sigma}(i)-X_{\sigma}(i-1)\leq X_{\sigma}(k_0)-X_{\sigma}(k_0-1)+1 \leq a_{k_{0}}.$$
Furthermore, we have 
\begin{eqnarray*}
X_{\sigma}(k')-X_{\sigma}(k_0)&=&(X_{\sigma}(k')-X_{\sigma}(k'-1))+(X_{\sigma}(k'-1)-X_{\sigma}(k'-2))+\dots\\
&& + (X_{\sigma}(k_0+2)-X_{\sigma}(k_0+1))+(X_{\sigma}(k_0+1)-X_{\sigma}(k_0))\\
&\leq&(k'-k_0)a_k.
\end{eqnarray*}
We get
$X_{\sigma}(k')=X_{\sigma}(k')-X_{\sigma}(k_0)+X_{\sigma}(k_0)<(k'-k_0)a_{k_{0}}+A_{k_{0}}+a_{k'}=A_{k'}.$\hfill$\Box$
\end{proof}

Note, that the proof of (+) can also be used to show the opposite direction, because the lexicographical sorting is not needed. The idea of Chen was to consider all $n$ inequalities of the form $\sum_{i=1}^{k}\min{(b_i,k-1)} \geq A_k -\sum_{i=k+1}^{n}\min{(b_i,k)}$ with $k \in \{1,\dots,n\}.$ He proved that the right-hand-side of these inequalities is maximized if $S$ is in decreasing lexicographical order. Unfortunately, by this approach he overlooked that other decreasing orders of $S$ with respect to the $a_i$ are also sufficient and the condition of a lexicographical order is too strong. On the other hand, if one wants to identify a threshold sequence with these $n$ inequalities, this can only be done with a lexicographical sorted sequence. In this case the inequalities have to fulfill equality for each $k.$

\subsection*{Acknowledgement.}
The author wishes to thank two anonymous referees. With their hint I was able to strengthen the main result similar to the result by Tripathi and Vijay \cite{TripathiVijay03} in the case of graphs. Furthermore, the proofs look more beautiful.

\providecommand{\bysame}{\leavevmode\hbox to3em{\hrulefill}\thinspace}
\providecommand{\MR}{\relax\ifhmode\unskip\space\fi MR }
\providecommand{\MRhref}[2]{%
  \href{http://www.ams.org/mathscinet-getitem?mr=#1}{#2}
}
\providecommand{\href}[2]{#2}

\end{document}